\documentclass[11pt]{article}
\usepackage{amsmath,amsthm,amssymb}
\usepackage[dvipsnames]{xcolor}

\usepackage[hyperfootnotes=false]{hyperref}

\usepackage{pxfonts}
\usepackage{mathrsfs}

\usepackage{titlesec}
\titleformat{\paragraph}
{\normalfont\normalsize\bfseries}{\theparagraph}{1em}{}
\titlespacing*{\paragraph}
{0pt}{3.25ex plus 1ex minus .2ex}{1.5ex plus .2ex}

\def\titlerunning#1{\gdef\titrun{#1}}
\makeatletter
\def\author#1{\gdef\autrun{\def\and{\unskip, }#1}\gdef\@author{#1}}
\def\address#1{{\def\and{\\\hspace*{18pt}}\renewcommand{\thefootnote}{}%
\footnote {#1}}%
\markboth{\autrun}{\titrun}}
\makeatother
\def\email#1{\hspace*{4pt}{\em e-mail}: #1}
\def\MSC#1{{\renewcommand{\thefootnote}{}%
\footnote{\emph{Mathematics Subject Classification (2020):} #1}}}
\def\keywords#1{\par\medskip
\noindent\textbf{Keywords:} #1}

\newtheorem{theorem}{Theorem}[section]
\newtheorem{prop}[theorem]{Proposition}
\newtheorem{cor}[theorem]{Corollary}
\newtheorem{lemma}[theorem]{Lemma}

\theoremstyle{definition}

\numberwithin{equation}{section}

\def\d{{\boldsymbol \delta}}

\def\0{\mathbf 0}

\def\cB{\mathcal B}

\def\cI{\mathcal I}

\def\cL{\mathcal L}
\def\cM{\mathcal M}

\def\cP{\mathcal P}
\def\cQ{\mathcal Q}

\def\cD{\mathcal D}

\def\cR{\mathcal R}
\def\cS{\mathcal S}

\def\PG{\mathrm{PG}}

\def\F{{\mathbb F}}

\def\N{{\mathrm N}}

\def\PGL{{\rm PGL}}

\def\GL{{\rm GL}}

\def\PGaL{{\rm P\Gamma L}}

\def\diag{{\rm diag}}
\def\rk{{\rm rk}}

\newcommand{\Fix}{\mathrm{Fix}}

\begin{document}


\baselineskip=16pt


\titlerunning{}

\title{Segre Varieties and Desarguesian Spreads}

\author{Antonio Cossidente \and Giuseppe Marino \and Francesco Pavese \and Paolo Santonastaso \and John Sheekey}

\date{}

\maketitle
\address{A. Cossidente: Dipartimento di Ingegneria, Università degli Studi della Basilicata, Contrada Macchia Romana, 85100, Potenza, Italy; \email{ antonio.cossidente@unibas.it}
}
\address{G. Marino: Dipartimento di Matematica e Applicazioni ``R. Caccioppoli'', Università degli Studi di Napoli, Via Cintia, Monte Sant'Angelo, 80126 Napoli, Italy; \email{giuseppe.marino@unina.it}
}
\address{F. Pavese: Dipartimento di Meccanica, Matematica e Management, Politecnico di Bari, Via Orabona 4, 70125 Bari, Italy; \email{francesco.pavese@poliba.it}
}
\address{P. Santonastaso: Dipartimento di Meccanica, Matematica e Management, Politecnico di Bari, Via Orabona 4, 70125 Bari, Italy \\
Dipartimento di Matematica e Applicazioni ``R. Caccioppoli'', Università degli Studi di Napoli, Via Cintia, Monte Sant'Angelo, 80126 Napoli, Italy; \email{paolo.santonastaso@poliba.it}
}
\address{J. Sheekey: School of Mathematics and Statistics, University College Dublin,  Belfield, Dublin 4, Ireland; \email{john.sheekey@ucd.ie}
}

\bigskip

\MSC{Primary 51E20; Secondary 05B25.}



\begin{abstract}
Let $\PG(n-1,q)$ denote the $(n-1)$-dimensional projective space over $\F_q$. We investigate the intersection of two Desarguesian $(h-1)$-spreads of $\PG(kh-1,q)$ and show that it is determined by a subgeometry over a suitable extension field. Our approach combines a characterization of subsets of points of $\PG(k-1,q^h)$ closed under $q$-order subgeometries with a matrix model for Desarguesian spreads based on Moore matrices. This leads naturally to the notion of generalized Segre varieties $\mathcal S^r_{kr-1,h-1}(q)$ and a geometric description of their maximal subspaces. As a main application, we prove that if two distinct Desarguesian $(h-1)$-spreads of $\PG(kh-1,q)$ contain a common pseudo-arc of size $k+1$, then their intersection is precisely the system $\cR^r_{h,q}$ of $(h-1)$-dimensional subspaces of $\mathcal S^r_{kr-1,h-1}(q)$, for some proper divisor $r$ of $h$.
\end{abstract}


\section{Introduction}


Let $\F_q$ be the finite field with $q$ elements, where $q$ is a prime power.
An $(h-1)$-spread of $\PG(n-1,q)$ is a partition of the point set into mutually
disjoint $(h-1)$-dimensional subspaces; such a spread exists if and only if
$h$ divides $n$.
A distinguished class of spreads is formed by the \emph{Desarguesian spreads}.
Their study goes back to the foundational work of B.~Segre~\cite{segre1964teoria}. These spreads play a central role in finite geometry, where they arise naturally
in connection with translation planes, blocking sets, linear sets, eggs, translation generalized quadrangles,
and several other incidence structures, see e.g.
\cite{HirschfeldThas,Lavrauw2015,bader2011desarguesian}. 

A useful description of a Desarguesian spread is
in terms of its \emph{director spaces}.
More precisely, if $\PG(kh-1,q)$ is embedded as a subgeometry of
$\PG(kh-1,q^h)$ and fixed pointwise by a semilinear collineation $\Psi$ of order $h$,
then a $(k-1)$-dimensional subspace $\Theta$, such that $\Theta,\Theta^{\Psi},\ldots,\Theta^{\Psi^{h-1}}$ span $\PG(kh-1,q^h)$, determines a spread whose elements are the intersections of
$\PG(kh-1,q)$ with the subspaces
$\langle P,P^\Psi,\dots,P^{\Psi^{h-1}}\rangle_{q^h}$,
where $P$ ranges over $\Theta$.
The subspaces $\Theta,\Theta^\Psi,\dots,\Theta^{\Psi^{h-1}}$ are called the
director spaces of the spread, and they uniquely determine it
\cite{bader2011desarguesian}. On the other hand, the Segre variety $\cS_{k-1,h-1}(q)$ is the image of the Segre
embedding
\[
\sigma_{k,h}:\PG(k-1,q)\times \PG(h-1,q)\longrightarrow \PG(kh-1,q),
\]
and consists of the points corresponding to rank-one $k\times h$ matrices.
Its geometry is governed by two families of maximal subspaces, the \textit{reguli},
whose elements are mutually disjoint within each family and intersect pairwise
in exactly one point across the two families.
Segre varieties thus provide the natural geometric model for rank-one tensors
and appear throughout the theory of field reduction, linear sets, and
semifields
\cite{HirschfeldThas,Lavrauw2015,Lavrauw2014segre}. A particularly significant link between these two structures is the following:
the elements of a Desarguesian spread corresponding to the points of a
canonical subgeometry inside a director space form precisely the
$(h-1)$-regulus of a Segre variety, see e.g. \cite[Theorem 2.6]{Lavrauw2015}.

The aim of this paper is to determine the intersection of two
Desarguesian spreads of $\PG(kh-1,q)$.
We show that such intersections are governed by subfield geometry and can be
described in terms of suitable generalized Segre varieties. Our first result is a characterization of subsets of points of
$\PG(k-1,q^h)$ that are closed with respect to $q$-order subgeometries.
This extends a theorem of Rottey and Van de Voorde for the projective line and
shows that such sets necessarily arise from subgeometries defined over a
subfield \cite{Rottey}. Next, we introduce the geometric setting through a matrix model for the Desarguesian spread. More precisely, we define the ambient projective space over $\F_{q^h}$ in terms of the vector space of $h\times k$
matrices with entries in $\F_{q^h}$. We endow this space with a suitable semilinear
collineation whose fixed-point set is a canonical subgeometry isomorphic to
$\PG(kh-1,q)$. In this model, the points of the fixed subgeometry are represented by Moore matrices and we describe the relation between its spread elements and its director
spaces. Within this framework, we introduce the notion of a \emph{generalized Segre
variety} $\mathcal S^r_{kr-1,h-1}(q)$, associated with a proper divisor $r$ of
$h$. More precisely, if $\theta_{q^r}\subset \Theta$ denotes the canonical
$q^r$-subgeometry of a director space, then the corresponding subset
$\cR^r_{h,q}$ of the Desarguesian spread $\cD_h$, consisting of the $(h-1)$-dimensional subspaces
associated with the points of $\theta_{q^r}$, forms one of the two natural
systems of maximal subspaces of $\mathcal S^r_{kr-1,h-1}(q)$. The generalized
Segre variety is defined as the union of the elements of $\cR^r_{h,q}$; moreover,
it is also partitioned by a second system $\cR^r_{kr,q}$ of mutually disjoint
$(kr-1)$-dimensional subspaces. In this sense, $\mathcal{R}^r_{h,q}$ and $\mathcal{R}^r_{kr,q}$ play the role of
the two reguli of the classical Segre variety.

Our main theorem determines the structure of the intersection of two distinct
Desarguesian $(h-1)$-spreads of $\PG(kh-1,q)$ containing a common pseudo-arc of
size $k+1$. Precisely, the common
pseudo-arc first determines a common $(h-1)$-regulus of a Segre variety; then, we prove that, via the matrix-geometric framework,  any additional common spread element gives rise to a proper
canonical $q^r$-subgeometry of a director space containing the canonical
$q$-subgeometry. At that point, by making use of
the subgeometry-closure theorem,
we show that the intersection is exactly the system $\cR^r_{h,q}$ of $(h-1)$-dimensional subspaces
of a generalized Segre variety $\cS^r_{kr-1,h-1}(q)$, for some proper divisor
$r$ of $h$.

The paper is organized as follows.
In Section~\ref{sec:preliminaries}, we recall the preliminaries on
Desarguesian spreads and Segre varieties, and we establish the characterization
of point sets closed under $q$-order subgeometries.
Section~\ref{sec:geometricsetting} develops the geometric framework, including
an explicit matrix model for the Desarguesian spread, the introduction of
generalized Segre varieties, and the study of their properties.
In Section~\ref{sec:intersection}, we determine the intersection structure of
two Desarguesian spreads and prove the main theorem.

\section{Preliminaries} \label{sec:preliminaries}

In this section, we collect some definitions and known results that will be
useful throughout the paper. In particular, we briefly recall Desarguesian spreads and Segre varieties, together with some of their basic properties. We also show a characterization result for sets of points having a closure property with respect to $q$-order subgeometries.

\subsection{Desarguesian spreads}

An $(h-1)$-spread of the projective space $\PG(n-1,q)$ is a collection
$\mathcal S$ of pairwise disjoint $(h-1)$-dimensional subspaces such that
every point of $\PG(n-1,q)$ is contained in exactly one element of
$\mathcal S$. Equivalently, an $(h-1)$-spread induces a partition of the
point set of $\PG(n-1,q)$ into $(h-1)$-dimensional subspaces.
Two spreads $\mathcal D_1$ and $\mathcal D_2$ of $\PG(n-1,q)$ are said to be
equivalent if there exists a collineation of $\PG(n-1,q)$ mapping
$\mathcal D_1$ onto $\mathcal D_2$.

A necessary and sufficient condition for the existence of an $(h-1)$-spread
in $\PG(n-1,q)$ is that $h$ divides $n$, as shown by Segre \cite{segre1964teoria}.
Accordingly, we write $n=hk$ for some positive integer $k$.
A relevant class of spreads is given by the \emph{Desarguesian spreads},  which can
be described as follows; see \cite{segre1964teoria}. Let $\Sigma \simeq \PG(kh-1,q)$ be embedded as a $q$-order subgeometry of
$\PG(kh-1,q^h)$ in such a way that $\Sigma$ coincides with
the set of fixed points
$
\Sigma = \Fix(\Psi)
$
of a semilinear collineation $\Psi$ of $\PG(kh-1, q^h)$ of order $h$.
For any point $P \in \PG(kh-1, q^h)$, define
\begin{align*}
X^*(P) := \langle P, P^{\Psi}, \ldots, P^{\Psi^{h-1}} \rangle_{q^h}.    
\end{align*}
Observe that $X^*(P)$ has dimension at most $h-1$ for every $P \in \PG(kh-1, q^h)$. Let $\Theta$ be a $(k-1)$-dimensional subspace of $\PG(kh-1, q^h)$ satisfying
\begin{equation}\label{eq:spanentiredesarguesian}
\langle \Theta, \Theta^{\Psi}, \ldots, \Theta^{\Psi^{h-1}} \rangle_{q^h}
= \PG(kh-1, q^h).
\end{equation}
Then, for every point $P \in \Theta$, Condition~\eqref{eq:spanentiredesarguesian}
implies that $X^*(P)$ is an $(h-1)$-dimensional subspace of $\PG(kh-1, q^h)$. By construction, each subspace $X^*(P)$ with $P \in \Theta$ is stabilized by
$\Psi$. Hence, by \cite[Lemma~1]{lunardon1999normal}, the intersection
\begin{align*}
X(P) := X^*(P) \cap \Sigma  
\end{align*}
is an $(h-1)$-dimensional subspace of $\Sigma$.
As $P$ ranges over $\Theta$, we obtain the family
\[
\cL(\Theta)
:= \{\, X(P) = X^*(P) \cap \Sigma : P \in \Theta \,\},
\]
which consists of $q^{h(k-1)} + q^{h(k-2)} + \cdots + q^h + 1$ mutually
disjoint $(h-1)$-dimensional subspaces of $\Sigma$, and therefore forms an
$(h-1)$-spread of $\Sigma$.

Any $(h-1)$-spread of $\PG(kh-1,q)$ that is $\PGL$-equivalent to a spread  $\cL(\Theta)$ is called a \emph{Desarguesian spread} of
$\PG(kh-1,q)$.
Moreover, the $(k-1)$-dimensional subspaces
$
\Theta, \Theta^{\Psi}, \ldots, \Theta^{\Psi^{h-1}}
$
are uniquely determined by the spread $\cL(\Theta)$.
More precisely, if $\cL(\Theta) = \cL(\Theta')$ for some
$(k-1)$-subspace $\Theta'$ of $\PG(kh-1, q^h)$, then necessarily
$\Theta' = \Theta^{\Psi^i}$ for some $i \in \{0,\ldots,h-1\}$.
For this reason, the subspaces
$\Theta, \Theta^{\Psi}, \ldots, \Theta^{\Psi^{h-1}}$
are referred to as the \emph{director spaces} of the Desarguesian spread
$\cL(\Theta)$.
Further details can be found in
\cite{bader2011desarguesian, lunardon1999normal, segre1964teoria, VdV}.

\subsection{Segre varieties}

Let $k,h\ge 1$. The \emph{Segre map}
\[
\sigma_{k,h} : \PG(k-1,q)\times \PG(h-1,q)
\longrightarrow \PG(kh-1,q)
\]
is defined by
\[
\left((x_1,\ldots,x_k),(y_1,\ldots,y_h)\right) \mapsto (x_1y_1,\ldots,x_1y_h,\ldots,x_ky_1,\ldots,x_ky_h),
\]
where the coordinates $x_i y_j$ are taken in lexicographical order.
The image of the Segre map $\sigma_{k,h}$ is an algebraic variety called the \emph{Segre variety} and it is denoted by
$\mathcal S_{k-1,h-1}(q)$.

\noindent We may represent the points of $\PG(kh-1,q)$ by the homogeneous coordinates \[
(x_{11},x_{12},\ldots,x_{1h};
    x_{21},\ldots,x_{2h};
    \ldots;
    x_{k1},\ldots,x_{kh}).
\]
In this way, to each point of $\PG(kh-1,q)$ we can naturally associate a
$k\times h$ matrix $(x_{ij})$, with $1\le i\le k$ and $1\le j\le h$.
Consider the family of quadrics whose defining equations are given by the
$2\times2$ minors of this matrix.
It can be shown that the Segre variety $\mathcal S_{k-1,h-1}(q)$ coincides with
the intersection of all quadrics defined in this way, see \cite[Theorem 4.94]{HirschfeldThas}.
Equivalently, the points of the Segre variety $\mathcal S_{k-1,h-1}(q)$ are precisely
those whose coordinates form a matrix $(x_{ij})$ of rank $1$, see \cite[Theorem 4.101]{HirschfeldThas}.

The variety $\mathcal S_{k-1,h-1}(q)$ contains
two distinguished families of maximal subspaces. Fixing a point of $\PG(k-1,q)$ and letting the point of $\PG(h-1,q)$ vary, we
obtain an $(h-1)$-dimensional subspace entirely contained in
$\mathcal S_{k-1,h-1}(q)$. As the point of $\PG(k-1,q)$ varies, the resulting
family of pairwise disjoint $(h-1)$-dimensional subspaces forms a so called
\emph{$(h-1)$-regulus} $\mathcal R_{h,q}$ of the Segre variety
$\mathcal S_{k-1,h-1}(q)$. Analogously, fixing a point of $\PG(h-1,q)$ and varying the point in
$\PG(k-1,q)$ yields a $(k-1)$-dimensional subspace of
$\mathcal S_{k-1,h-1}(q)$. These subspaces form another family, called a
\emph{$(k-1)$-regulus} $\cR_{k, q}$ of $\mathcal S_{k-1,h-1}(q)$. Distinct subspaces belonging to the same regulus are mutually disjoint, whereas any two
subspaces from different reguli intersect in exactly one point. Moreover, every
maximal subspace contained in $\mathcal S_{k-1,h-1}(q)$ belongs to exactly
one of these two reguli. For further details, we refer the reader to
\cite{harris2013algebraic,HirschfeldThas}.

\subsection{A preliminary result}

In this section, we investigate a higher-dimensional analogue of \cite[Theorem 1.5]{Rottey}. More precisely, we study subsets of points of 
$\PG(k-1,q^h)$ satisfying a closure property with respect to 
$q$-order subgeometries, and we show that such sets necessarily arise from 
a subgeometry over a subfield. In \cite{Rottey} the authors prove the following.

\begin{theorem}[see {\cite[Theorem 1.5]{Rottey}}] \label{rottey}
Assume that $q>2$. If $\cS$ is a set of at least three points of $\PG(1,q^h)$, such that any three points of $\cS$ determine a $q$-order subline entirely contained in $\cS$, then $\cS$ defines a $q^r$-order subline of $\PG(1, q^h)$ for
some divisor $r$ of $h$.    
\end{theorem}

A set of $k+1$ points of $\PG(k-1,q^h)$ is said to be in {\em general position} if no $k$ of them lie in a hyperplane. Note that $k+1$ points of $\PG(k-1, q^h)$ in general position determine a unique $q$-order subgeometry of $\PG(k-1, q^h)$. Here, we are interested in a generalization of Theorem \ref{rottey}. Precisely, we consider $\cS$ to be a set containing $k+1$ points of  
$\PG(k-1,q^h)$, with $q>2$, in general position, such that for any choice 
of $k+1$ points of $\mathcal S$ in general position, the unique $q$-order 
subgeometry determined by them is entirely contained in $\cS$.

The following lemmas collect some geometric properties that follow from
the assumptions on $\mathcal S$.

\begin{lemma} \label{lemma:retta1}
Let $\cS$ be a subset of points of $\PG(k-1,q^h)$ such that $\cS$ contains $k+1$ points in general position, and such that for any 
$k+1$ points of $\cS$ in general position, the unique $q$-order 
subgeometry they determine is entirely contained in $\cS$. If a line $\ell$ has at least two points in common with $\cS$, then $|\ell \cap \cS| \ge q+1$. Also, if $\ell_1$ and $\ell_2$ are two intersecting lines such that $|\cS \cap \ell_i| \ge 2$, $i = 1,2$, then $\ell_1 \cap \ell_2 \in \cS$. 
\end{lemma}
\begin{proof}
First, we prove that if a line $\ell$ has at least two points in common with 
$\mathcal S$, then $|\ell \cap \mathcal S| \ge q+1$. 
Let $P_1$ and $P_2$ be distinct points of $\mathcal S$, and let $\ell$ be the line 
of $\PG(k-1,q^h)$ through them. Observe that $\cS$ contains a $q$-order subgeometry of $\PG(k-1, q^h)$. Hence we can find  $P_3,\dots,P_{k+1}$ further $k-1$ points 
of $\mathcal S$ such that no $k$ points of $\{P_1,\dots,P_{k+1}\}$ lie in a 
hyperplane. Then the $q$-order subgeometry determined by $P_1,\dots,P_{k+1}$ 
contains $q+1$ points of $\ell$ and, by hypothesis, is entirely contained in 
$\mathcal S$. Hence, if a line $\ell$ has at least two points in common with 
$\mathcal S$, then $|\ell \cap \mathcal S| \ge q+1$, proving the first part of 
the assertion.

\noindent To prove the second part, let $\ell_1$ and $\ell_2$ be two intersecting lines such 
that $|\mathcal S \cap \ell_i| \ge 2$ for $i=1,2$. Let $P_1,P_2 \in \ell_1 \cap 
\mathcal S$ and let $Q_1,Q_2 \in \ell_2 \cap \mathcal S$. If one of the points 
$P_1,P_2$ lies on $\ell_2$, or one of the points $Q_1,Q_2$ lies on $\ell_1$, then 
$\ell_1 \cap \ell_2 \in \mathcal S$ and the assertion follows immediately. 
Therefore, we may assume that this is not the case.

\noindent Denote by $\theta_q$ the $q$-order subgeometry determined by $k+1$ points of 
$\mathcal S$ in general position; in particular, $\theta_q \subseteq \mathcal S$. 
Let $\Lambda$ be a $(k-4)$-dimensional subspace of $\PG(k-1,q^h)$ disjoint from 
$\langle \ell_1,\ell_2 \rangle_{q^h} \simeq \PG(2,q^h)$ such that 
$\Lambda \cap \theta_q \simeq \PG(k-4,q)$. Let $T_1,\dots,T_{k-2}$ be $k-2$ points 
of $\Lambda \cap \theta_q$ such that no $k-3$ of them lie in a $(k-5)$-dimensional subspace of 
$\Lambda$. Since $T_{k-2}$ and $Q_2$ belong to $\mathcal S$, by the first part of the proof 
the line $\langle T_{k-2},Q_2 \rangle_{q^h}$ contains at least $q+1$ points of 
$\mathcal S$. Hence, we may choose a point 
$T' \in \langle T_{k-2},Q_2 \rangle_{q^h} \cap \mathcal S \setminus 
\{T_{k-2},Q_2\}$. Consider the $k+1$ points
\[
T_1,\dots,T_{k-3},T',P_1,P_2,Q_1,
\]
which are in general position. Let $\theta_q'$ be the $q$-order subgeometry 
determined by these points. By hypothesis, $\theta_q' \subseteq \mathcal S$. Moreover, $\Lambda \cap \theta_q' \simeq \PG(k-4,q)$ and 
$\langle T',P_1,P_2,Q_1 \rangle_{q^h} \cap \theta_q' \simeq \PG(3,q)$; hence,
\[
\Lambda \cap \langle T',P_1,P_2,Q_1 \rangle_{q^h} = T_{k-2} \in \theta_q'.
\]
Similarly, 
$\langle T_{k-2},T' \rangle_{q^h} \cap \theta_q' \simeq \PG(1,q)$ and 
$\langle P_1,P_2,Q_1 \rangle_{q^h} \cap \theta_q' \simeq \PG(2,q)$, which implies
\[
\langle T_{k-2},T' \rangle_{q^h} \cap 
\langle P_1,P_2,Q_1 \rangle_{q^h} = Q_2 \in \theta_q'.
\]

\noindent Therefore, $\ell_1$ and $\ell_2$ are two extended lines of the subgeometry 
$\theta_q'$, and it follows that \[\ell_1 \cap \ell_2 \in \theta_q' \subseteq 
\mathcal S,\] completing the proof.
\end{proof}

\begin{lemma} \label{lemma:retta2}
Let $\cS$ be a subset of points of $\PG(k-1,q^h)$ such that $\cS$ contains $k+1$ points in general position, and such that for any 
$k+1$ points of $\cS$ in general position, the unique $q$-order 
subgeometry they determine is entirely contained in $\cS$. Then a line of $\PG(k-1, q^h)$ meets $\cS$ in $0$, $1$, or $q^r+1$ points, for some divisor $r$ of $h$.    
\end{lemma} 
\begin{proof}
Denote by $\theta_q$ the $q$-order subgeometry determined by $k+1$ points of $\cS$ in general position. Hence, $\theta_q \subseteq \cS$. By Lemma \ref{lemma:retta1}, we know that if a line $\ell$ has at least two points in common with $\mathcal S$, then $|\ell \cap \mathcal S| \ge q+1$. We start by proving that the $q$-order subline determined by three distinct points 
$P_1$, $P_2$, and $R$ of $\ell \cap \mathcal S$ is entirely contained in $\ell \cap \mathcal S$. 
Let $\Gamma$ be a $(k-3)$-dimensional subspace of $\PG(k-1, q^h)$ disjoint from $\ell$ such that $\Gamma \cap \theta_q \simeq \PG(k-3, q)$. Let $R_1, \dots, R_{k-1}$ be a set of $k-1$ points of $\Gamma \cap \theta_q$ such that no $k-2$ of them are in a $(k-4)$-dimensional subspace of $\Gamma$. Consider a point $R'$ belonging to the line joining $R_{k-1}$ and $R$, with $R' \in \cS \setminus \{R, R_{k-1}\}$. Then $R_1, \dots, R_{k-2}, R', P_1, P_2$ are $k+1$ points of $\cS$ in general position, and $\theta_q' \subseteq \cS$, where $\theta_q'$ is the $q$-order subgeometry  determined by $R_1, \dots, R_{k-2}, R', P_1, P_2$. Moreover, $\Gamma \cap \theta_q' \simeq \PG(k-3, q)$, $\langle R', P_1, P_2 \rangle_{q^h} \cap \theta_q' \simeq \PG(2, q)$, and hence $\Gamma \cap \langle R', P_1, P_2 \rangle_{q^h} = R_{k-1} \in \theta_q'$. Similarly, $\langle R_{k-1}, R' \rangle_{q^h} \cap \theta_q' \simeq \PG(1, q)$, $\langle P_1, P_2 \rangle_{q^h} \cap \theta_q' \simeq \PG(1, q)$, and hence $\langle R_{k-1}, R' \rangle_{q^h} \cap \langle P_1, P_2 \rangle_{q^h} = R \in \theta_q'$. It follows that the $q$-order subline determined by $P_1, P_2, R$ is entirely contained in $\ell \cap \cS$. Therefore, if $\ell$ is a line such that $|\ell \cap \cS| \ge 2$, by Theorem~\ref{rottey}, there exists a divisor $r_\ell$ of $h$ such that $|\ell \cap \cS| = q^{r_\ell}+1$. 

Finally, we show by induction on $k$ that the integer $r_\ell$ does not depend on $\ell$.
If $k = 2$, then $\PG(k-1, q^h)$ is a line $\ell$ and $|\ell \cap \cS| = q^r+1$. Assume $k\ge 3$ and suppose, by induction, that in $\Lambda \simeq \PG(k-2,q^h)$ every line meeting
$\cS\cap \Lambda$ in at least two points meets it in exactly $q^{r}+1$ points, for some divisor $r$ of $h$.
Choose such a hyperplane $\Lambda$ with 
\[
|\cS\cap \Lambda|\ge \frac{q^{k-1}-1}{q-1},
\]
and fix a point $P\in \cS\setminus \Lambda$. Let $s_1$ and $s_2$ be distinct lines through $P$ such that $|s_i\cap \cS|\ge 2$, $i=1,2$. Then $|s_i \cap \cS| = q^{r_{s_i}} + 1$, for some divisors $r_{s_1}$ and $r_{s_2}$ of $h$. Set $X_i = s_i \cap \PG(k-2, q^h)$, $i = 1, 2$. Arguing as before, it can be seen that both $X_1$ and $X_2$ belong to $\cS$. Hence $\left|\langle X_1, X_2 \rangle_{q^h} \cap \cS\right| = q^r+1$. Let $Q$ be a point of $\cS$ on the line $\langle X_1, X_2 \rangle_{q^h}$ distinct from $X_1$ and $X_2$. Note that if $Y_1 \in s_1 \cap \cS$, then $Y_2 = \langle Q, Y_1 \rangle_{q^h} \cap s_2 \in \cS$, by Lemma \ref{lemma:retta1}. Hence the map 
\begin{align*}
Y_1 \in s_1 \cap \cS \mapsto \langle Q, Y_1 \rangle_{q^h} \cap s_2 \in s_2 \cap \cS    
\end{align*}
establishes a bijection between $s_1 \cap \cS$ and $s_2 \cap \cS$. This implies that there exists an integer $r_0$,
depending only on $P$, such that
\[
|s\cap \mathcal S| = q^{r_0}+1
\]
for every line $s$ through $P$ with $|s\cap \mathcal S|\ge 2$. Now let $\Lambda' \simeq \PG(k-2,q^h)$ be a hyperplane containing the point
$P$ and a point $Q' \in \Lambda \cap \cS$, with
\[
|\cS \cap \Lambda' | \ge \frac{q^{k-1}-1}{q-1}.
\] Let $P' \in (\Lambda \cap \cS) \setminus \Lambda'$.
By the previous argument, we get that the line $\langle P,Q'\rangle_{q^h}$ meets $\mathcal S$ in exactly
$q^{r_0}+1$ points, while the line $\langle P',Q'\rangle_{q^h}$, which is
contained in $\Lambda$, meets $\mathcal S$ in exactly $q^{r}+1$ points by the
induction hypothesis. Since both lines contain the point $Q' \in \mathcal S$,
it follows that
\[
q^{r_0}+1 = q^{r}+1,
\]
and hence $r_0 = r$. By repeating the same argument for each point $P \in \mathcal S \setminus \Lambda$, the assertion follows.
\end{proof}

Combining the previous lemmas, we obtain the following characterization,
which generalizes Theorem \ref{rottey} to $\PG(k-1,q^h)$.

\begin{theorem} \label{thm:general}
Assume that $q>2$. Let $\cS$ be a subset of points of $\PG(k-1,q^h)$ such that $\cS$ contains $k+1$ points in general position, and such that for any 
$k+1$ points of $\cS$ in general position, the unique $q$-order 
subgeometry they determine is entirely contained in $\cS$. Then $\cS$ defines a $q^r$-order subgeometry of $\PG(k-1, q^h)$ for
some divisor $r$ of $h$.    
\end{theorem}
\begin{proof}
Consider the incidence structure whose set of points is $\cS$, whose set of blocks is
\begin{align*}
    \cB = \left\{\ell \cap \cS \mid \ell \mbox{ line of } \PG(k-1, q^h), |\ell \cap \cS| = q^r+1\right\},
\end{align*}
and incidence is given by containment. Then $(\cS, \cB)$ satisfies the following properties:
\begin{itemize}
    \item[1)] For any two distinct points $P$ and $Q$  of $\cS$ there is exactly one block of $\cB$ incident with both $P$ and $Q$.
    \item[2)] Let $A, B, C, D$ be four points of $\cS$ such that the block containing $A$ and $B$ intersects the block containing $C$ and $D$. Then the block containing $A$ and $C$ also intersects the block containing $B$ and $D$. 
    \item[3)] Any block is incident with at least three points.
    \item[4)] There are at least two lines. 
\end{itemize}
Moreover, the theorem of Desargues holds in $(\cS, \cB)$. The statement follows, see for instance \cite{Beutelspacher}.
\end{proof}

\section{The geometric setting} \label{sec:geometricsetting}

This section is devoted to describing the geometric setting underlying our results. In this
framework, we investigate the interaction between Segre varieties and
Desarguesian spreads, and we introduce the notion of a generalized Segre
variety. We start by presenting a Desarguesian spread arising from a
matrix representation.

For a Segre variety $\cS_{k-1, h-1}(q)$ denote by $\cR_{k, q}$ and $\cR_{h, q}$ the $(k-1)$-regulus and the $(h-1)$-regulus of $\cS_{k-1, h-1}(q)$, respectively. Let $\overline{\cM}  = \cM_{h\times k}(\F_{q^h})$. Then $\PG(\overline{\cM} , q^h) = \PG(kh-1,q^h)$. Here we denote the rows of an element of $\overline{\cM} $ by vectors $v_1, \ldots, v_{h} \in \F_{q^h}^k$, and the vector $v^\sigma$ denotes the vector of $\F_{q^h}^k$ obtained by raising every entry of $v$ to the $q$-th power. Let $\psi$ denote the element of $\PGaL(kh, q^h)$ induced by
\begin{align*}
\begin{pmatrix}
v_1 \\
v_2 \\
\vdots \\
v_{h}
\end{pmatrix} \in \cM_{h \times k}(\F_{q^h})
\mapsto
\begin{pmatrix}
v_{h}^\sigma \\
v_1^\sigma \\
\vdots \\
v_{h-1}^\sigma
\end{pmatrix} \in \cM_{h \times k}(\F_{q^h}).
\end{align*}
Then $\psi$ is semilinear, has order $h$ and fixes pointwise the $q$-order subgeometry 
\begin{align*}
\Pi = \PG(\cM, q^h) \simeq \PG(kh-1, q), 
\end{align*}
where $\cM$ is the $kh$-dimensional $\F_q$-vector space of Moore matrices given by  
\begin{align*}
\cM = \left\{M_v \in \overline{\cM}  \mid v \in \F_{q^h}^k\right\}, \mbox{ with } 
M_v = 
\begin{pmatrix}
v \\
v^\sigma \\
\vdots \\
v^{\sigma^{h-1}} 
\end{pmatrix}.
\end{align*}
Let $\Theta$ denote the $(k-1)$-dimensional subspace of $\PG(kh-1, q^h)$ defined by the $k$-dimensional vector $\F_{q^h}$-subspace of $\overline{\cM} $ consisting of matrices whose only nonzero entries occur in the first row. Then $\Theta^{\psi^i}$ is defined by the $k$-dimensional vector $\F_{q^h}$-subspace of $\overline{\cM} $ consisting of matrices whose only nonzero entries occur in the $(i+1)$-th row, $i = 1, \dots, h-1$. Then $\Theta, \Theta^{\psi}, \ldots, \Theta^{\psi^{h-1}}$ span the whole $\PG(kh-1, q^h)$. Let $P$ be the point of $\Theta$ represented by 
\begin{align}
\begin{pmatrix}
v \\
\0 \\
\vdots \\
\0
\end{pmatrix}, \mbox{ where } v \in \F_{q^h}^k \setminus \{\0\}. \label{rep_P}
\end{align}
Then $\left\langle P, P^\psi, \dots, P^{\psi^{h-1}} \right\rangle_{q^h}$ is the $(h-1)$-dimensional subspace of $\PG(kh-1, q^h)$ defined by 
\begin{align}
\overline{\pi} _{v} = \left\{ \begin{pmatrix}
\alpha_1 v \\
\alpha_2 v^\sigma \\
\vdots \\
\alpha_h v^{\sigma^{h-1}}
\end{pmatrix} \mid \alpha_i \in \F_{q^h} \right\} = 
\left\{\diag(\alpha_1, \alpha_2, \dots, \alpha_{h}) M_v \mid  \alpha_i \in \F_{q^h} \right\} = 
\left\{\d M_v \mid \d \in \Delta\right\},  \label{rep_pi}
\end{align}
where $\Delta$ denotes the set of diagonal matrices whose entries are in $\F_{q^h}$. Since $\left\langle P, P^\psi, \dots, P^{\psi^{h-1}} \right\rangle_{q^h}$ is fixed by $\psi$, by \cite[Lemma 1]{lunardon1999normal}, it follows that $\Pi \cap \left\langle P, P^\psi, \dots, P^{\psi^{h-1}} \right\rangle_{q^h}$ is an $(h-1)$-dimensional subspace of $\Pi$. In particular, it is defined by 
\begin{align*}
\pi_v = \left\{ M_{\alpha v} \mid \alpha \in \F_{q^h} \setminus \{0\} \right\}.
\end{align*}
As $P$ varies in $\Theta$, we obtain the set $\cD_h$, which consists of $\frac{q^{hk} - 1}{q^{h} - 1}$ mutually disjoint $(h-1)$-dimensional subspaces of $\Pi$. Hence, $\mathcal D_h$ is an $(h-1)$-spread of $\Pi$ and, by construction, a
Desarguesian $(h-1)$-spread, which we denote by
$\mathcal D_h = \mathcal L(\Theta)$.
In this case, the $(k-1)$-dimensional subspaces
$\Theta, \Theta^{\psi}, \ldots, \Theta^{\psi^{h-1}}$
are the director spaces of $\mathcal D_h$.

If $v \in \F_{q}^k \setminus \{\0\}$, the point $P$ of $\Theta$ represented by~\eqref{rep_P} lies in the canonical $q$-order subgeometry of $\Theta$, say $\theta_q$. Then, varying $v \in \F_{q}^k \setminus \{\0\}$, we obtain a subset of $\cD_h$ consisting of $\frac{q^{k}-1}{q-1}$ elements that form the $(h-1)$-regulus $\cR_{h, q}$ of a Segre variety $\cS_{k-1, h-1}(q)$, see \cite[Theorem 2.4]{Lavrauw2015}. Hence, by construction, the following holds true.
\begin{lemma}\label{intersection}
The set of intersections of the extensions of elements of $\cR_{h, q}$ with $\Theta$ is the canonical $q$-order subgeometry $\theta_{q}$.
\end{lemma}

Recall that a Segre variety $\cS_{k-1, h-1}(q) \subset \PG(kh-1, q)$ is the intersection of certain quadrics of $\PG(kh-1, q)$, see \cite[Theorem 4.94]{HirschfeldThas}. Then the intersection of the same set of quadrics considered in $\PG(kh-1, q^s)$ is a Segre variety $\cS_{k-1, h-1}(q^s)$. Hence, $\cS_{k-1, h-1}(q)$ uniquely extends over $\F_{q^s}$ to $\cS_{k-1, h-1}(q^s)$. Therefore, the extension of the Segre variety $\cS_{k-1, h-1}(q)$ over $\F_{q^h}$, denoted by $\cS_{k-1,h-1}(q^h)$, consists of the matrices of $\overline{\cM}$ having rank one. Moreover, $\cS_{k-1, h-1}(q^h)$ is left invariant by $\psi$ and $\Pi \cap \cS_{k-1, h-1}(q^h) = \cS_{k-1, h-1}(q)$. Its point set is partitioned by the $(h-1)$-regulus $\cR_{h, q^h}$ containing the extensions of the elements of $\cR_{h, q}$. In what follows it is pointed out that $\Theta, \Theta^\psi, \dots, \Theta^{\psi^{h-1}} \in \cR_{k, q^h}$.

\begin{lemma}\label{subgeometry}
Let $\cQ$ be a quadric of $\PG(k-1, q^h)$ such that a $q$-order subgeometry of $\PG(k-1, q^h)$ is contained in $\cQ$, then $\PG(k-1, q^h) \subset \cQ$. 
\end{lemma}
\begin{proof}
Since the group $\PGL_k(q^h)$ acts transitively on its $q$-order subgeometries, we may assume that the $q$-order subgeometry of $\PG(k-1, q^h)$ contained in $\cQ$ is $\PG(k-1, q)$ in canonical position. Let $F = \sum_{i,j} a_{i j} X_i X_j$ be the quadratic form defining $\cQ$. Since the point having $1$ in the $i$-th position and zero elsewhere belongs to $\cQ$, then $a_{i i} = 0$. Similarly, the point having $1$ in the $i$-th and $j$-th positions, $i \ne j$, and zero elsewhere is in $\cQ$. Hence $F$ is the null form, as required.
\end{proof}

\begin{lemma}\label{contained}
The director spaces $\Theta, \Theta^{\psi}, \dots, \Theta^{\psi^{h-1}}$ of $\cD_h$ belong to $\mathcal R_{k,q^h}$.
\end{lemma}
\begin{proof}
Since $\cS_{k-1, h-1}(q^h)$ is the intersection of certain quadrics, say  $\cQ_1, \dots, \cQ_{m}$, and $\theta_q \subset \cQ_i$, $i = 1, \dots, m$, it is enough to observe that $\Theta \subset \cQ_i$, $i = 1, \dots, m$, by Lemma~\ref{subgeometry}. Moreover, $\psi$ fixes $\cS_{k-1, h-1}(q^h)$ and hence $\Theta^{\psi^i} \subseteq \mathcal S_{k-1,h-1}(q^h)$, $i = 1, \dots, h-1$. By construction, every element of $\mathcal R_{h,q^h}$, that is the extension of
an element of $\mathcal R_{h,q}$, meets each of
$\Theta, \Theta^{\psi}, \dots, \Theta^{\psi^{h-1}}$ in exactly one point. Therefore $\Theta^{\psi^i} \in \cR_{k, q^h}$, $i = 0, \dots, h-1$.
\end{proof}

\subsection{The generalized Segre variety $\cS^r_{kr-1,h-1}(q)$}\label{sec:generalized-segre}

Let $r$ be a proper divisor of $h$ and set $h = rt$. If $v \in \F_{q^r}^k \setminus \{\0\}$, the point $P$ of $\Theta$ represented by~\eqref{rep_P} lies in the canonical $q^r$-order subgeometry of $\Theta$, say $\theta_{q^r}$. In this case, $\pi_v$ consists of 
\begin{align*}
M_{\alpha v} = 
\begin{pmatrix}
\alpha v \\
\alpha^q v^{\sigma} \\
\vdots \\
\alpha^{q^{r-1}} v^{\sigma^{r-1}} \\
\alpha^{q^{r}} v \\
\alpha^{q^{r+1}} v^{\sigma} \\
\vdots \\
\alpha^{q^{h-1}} v^{\sigma^{r-1}} 
\end{pmatrix}, \alpha \in \F_{q^h} \setminus \{0\}.
\end{align*}
Denote by $\cR^r_{h, q}$ the set of $\frac{q^{kr}-1}{q^r-1}$ members of $\cD_h$ defined by 
\begin{align*}
\pi_v, \mbox{ with } v \in \F_{q^r}^k \setminus \{\0\}.
\end{align*}
Let the {\em generalized Segre variety} $\cS^r_{kr-1,h-1}(q)$ be the set of points of $\Pi \simeq \PG(kh-1, q)$ lying in the union of the elements of $\cR^r_{h, q}$. In particular,
\begin{align*}
\left|\cS^r_{kr-1, h-1}(q)\right| = \frac{(q^h-1)(q^{kr}-1)}{(q-1)(q^r-1)}.
\end{align*}
Next, we show that the generalized Segre variety $\cS^r_{kr-1, h-1}(q)$, which has been defined from $\PG(k-1, q^r) \simeq \theta_{q^r} \subset \Theta \simeq \PG(k-1, q^h)$ by applying the field reduction $\PG(k-1,q^h) \mapsto \PG(kh-1, q)$, can also be obtained from $\cS_{k-1, t-1}(q^r) \subset \PG(kt-1, q^r)$ by applying the field reduction $\PG(kt-1,q^r) \mapsto \PG(kh-1, q)$. 
To this end, observe that $\psi^r$ is semilinear, has order $t$, and fixes pointwise the $q^r$-order subgeometry 
\begin{align*}
\Pi_r = \PG(\cM_r, q^r) \simeq \PG(kh-1, q^r), 
\end{align*}
where $\cM_r$ is the $kh$-dimensional $\F_{q^r}$-vector space 
\begin{align*}
\cM_r = \left\{M_{v_1, \dots, v_r} \in \overline{\cM}  \mid v_1, \dots, v_r \in \F_{q^h}^k\right\}, \mbox{ with } 
M_{v_1, \dots, v_r} = 
\begin{pmatrix}
v_1 \\
\vdots \\
v_r \\
v_1^{\sigma^r} \\
\vdots \\
v_r^{\sigma^r} \\
\vdots \\
v_1^{\sigma^{r(t-1)}} \\
\vdots \\
v_r^{\sigma^{r(t-1)}}  
\end{pmatrix}.
\end{align*}
In particular $\cM \subset \cM_r$ and hence $\Pi \subset \Pi_r$. Moreover, $\left\langle \Theta, \Theta^{\psi^r}, \dots, \Theta^{\psi^{r(t-1)}} \right\rangle_{q^h} \simeq \PG(kt-1, q^h)$ is fixed by $\psi^r$, from which, by \cite[Lemma 1]{lunardon1999normal}, we infer that 
\begin{align*}
\Pi_r \cap \left\langle \Theta, \Theta^{\psi^r}, \dots, \Theta^{\psi^{r(t-1)}} \right\rangle_{q^h} \simeq \PG(kt-1, q^r).
\end{align*}
More precisely, 
\begin{align*}
& \Theta_r, \Theta_r^\psi, \dots, \Theta_r^{\psi^{r-1}}, & \mbox{ where } \Theta_r = \Pi_r \cap \left\langle \Theta, \Theta^{\psi^r}, \dots, \Theta^{\psi^{r(t-1)}} \right\rangle_{q^h} \simeq \PG(kt-1, q^r), 
\end{align*}
are the director spaces of a Desarguesian $(r-1)$-spread of $\Pi$, say $\cD_{r}$. Recall that 
\begin{align*}
& \left\langle P, P^\psi, \dots, P^{\psi^{h-1}} \right\rangle_{q^h} \simeq \PG(h-1, q^h), & \mbox{ where } P \in \Theta, 
\end{align*}
are the elements of $\cD_h$ extended over $\F_{q^h}$. Then 
\begin{align*}
\Theta_r \cap  \left\langle P, P^\psi, \dots, P^{\psi^{h-1}} \right\rangle_{q^h} \simeq \PG(t-1, q^r).
\end{align*}
Hence $\Pi \cap \left\langle P, P^\psi, \dots, P^{\psi^{h-1}} \right\rangle_{q^h}$ is partitioned by members of $\cD_r$. On the other hand, $\psi^r$ stabilizes $\left\langle \Theta, \Theta^{\psi^r}, \dots, \Theta^{\psi^{r(t-1)}} \right\rangle_{q^h}$ and fixes pointwise $\Theta_r$; therefore, 
\begin{align*}
& \Theta, \Theta^{\psi^r}, \dots, \Theta^{\psi^{r(t-1)}}, 
\end{align*}
are the director spaces of a Desarguesian $(t-1)$-spread of $\Theta_r \simeq \PG(kt-1, q^r)$, say $\overline{\cD}_{t}$, where the members of $\overline{\cD}_t$ are given by 
\begin{align*}
& \Theta_r \cap \left\langle P, P^\psi, \dots, P^{\psi^{h-1}} \right\rangle_{q^h}, & \mbox{ where } P \in \Theta.
\end{align*}
Varying $P \in \theta_{q^r}$, we get a subset of $\overline{\cD}_t$ consisting of $\frac{q^{kr}-1}{q^r-1}$ elements forming the $(t-1)$-regulus $\cR_{t, q^r}$ of a Segre variety $\cS_{k-1, t-1}(q^r)$. Note that $\overline{\Gamma}$ the extension over $\F_{q^h}$ of an element $\Gamma \in \cR^r_{h, q}$ meets $\Theta_r$ in a member of $\cR_{t, q^r}$. Analogously, by considering the $(k-1)$-regulus $\cR_{k, q^r}$ of $\cS_{k-1, t-1}(q^r)$, we can deduce the existence of a set $\cR^r_{kr, q}$ consisting of $\frac{q^{h}-1}{q^r-1}$ mutually disjoint $(kr-1)$-dimensional subspaces of $\Pi$ partitioning $\cS^r_{kr-1, h-1}(q)$. For $\Gamma \in \cR^r_{h, q}$ 
\begin{align*}
\left\{\Gamma \cap \Gamma' \mid \Gamma' \in \cR^r_{kr, q}\right\}
\end{align*}
is a Desarguesian $(r-1)$-spread of $\Gamma$; similarly, for $\Lambda \in \cR^r_{kr, q}$
\begin{align*}
\left\{\Lambda \cap \Lambda' \mid \Lambda' \in \cR^r_{h, q}\right\}
\end{align*}
is a Desarguesian $(r-1)$-spread of $\Lambda$.

\subsection{A uniqueness result}\label{sec:unique}

Consider the $(h-1)$-regulus $\cR_{h, q}$ of a Segre variety $\cS_{k-1, h-1}(q)$ of $\Pi$. Let $\Xi$ denote a further $(h-1)$-dimensional subspace of $\Pi$ disjoint from $\cS_{k-1, h-1}(q)$ and $\overline{\Xi}$ its extension over $\F_{q^h}$. Here it is shown that, under the hypothesis that there is no $r$, proper divisor of $h$, such that $\theta_{q^r}$ contains $\overline{\Xi} \cap \Theta$ and $\theta_q$, there exists a unique Desarguesian $(h-1)$-spread $\cD_h = \cL(\Theta)$ of $\Pi$ containing both $\cR_{h, q}$ and $\Xi$.

\begin{lemma}\label{crucial}
The $(h-1)$-regulus $\cR_{h, q}$ of $\cS_{k-1, h-1}(q)$ is contained in $\cL(\Theta)$.
\end{lemma}
\begin{proof}
By Lemma \ref{contained}, we have that $\Theta,\Theta^\psi, \ldots, \Theta^{\psi^{h-1}}$ are elements of the system $\cR_{k, q^h}$ of $\cS_{k-1, h-1}(q^h)$. Then every member of $\cR_{h, q^h}$ has a point in common with $\Theta^{\psi^i}$, $i = 0, \dots, h-1$. Hence, the $(h-1)$-regulus $\cR_{h, q}$ of $\cS_{k-1, h-1}(q)$ is contained in $\cL(\Theta)$. 
\end{proof}

\begin{prop}\label{prop:point}
Let $\Xi$ be an $(h-1)$-dimensional subspace of $\Pi \simeq \PG(kh-1, q)$ disjoint from $\cS_{k-1, h-1}(q)$, and $\overline{\Xi}$ its extension in $\PG(kh-1, q^h)$. Then there exists a Desarguesian $(h-1)$-spread of $\Pi$ containing $\Xi$ and the $(h-1)$-regulus $\cR_{h, q}$ of $\cS_{k-1, h-1}(q)$ if and only if there exists a point $P$ of $\cS_{k-1, h-1}(q^h)$ such that $\overline{\Xi} = \left\langle P, P^\psi, \dots, P^{\psi^{h-1}} \right\rangle_{q^h}$. Moreover, $P, P^\psi, \dots, P^{\psi^{h-1}} \in \cS_{k-1, h-1}(q^h)$.
\end{prop}
\begin{proof}
If $\cL(\Theta)$ is a Desarguesian $(h-1)$-spread of $\Pi$ containing $\cR_{h, q}$, then by Lemma~\ref{crucial}, $\Theta^{\psi^i} \in \cR_{k, q^h}$. Hence, if $\Xi \in \cL(\Theta)$, let $P^{\psi^i} = \overline{\Xi} \cap \Theta^{\psi^i}$, $i = 0, \dots, h-1$. In particular, $\overline{\Xi}$ is spanned by $P, P^\psi, \dots, P^{\psi^{h-1}}$ and $P^{\psi^i} \in \overline{\Xi} \cap \cS_{k-1, h-1}(q^h)$, $i = 0, \dots, h-1$. Viceversa, if $P \in \cS_{k-1, h-1}(q^h)$, let $\Theta$ be the unique member of $\cR_{k, q^h}$ through $P$. Then $\Theta^{\psi^i} \in \cR_{k, q^h}$, $i = 0, \dots, h-1$. Observe that $\Xi$ belongs to $\cL(\Theta)$ by construction, and the system $\cR_{h, q}$ of $\cS_{k-1, h-1}(q)$ is contained in $\cL(\Theta)$ by Lemma~\ref{crucial}.
\end{proof}

\begin{cor}\label{cor:point}
Let $\Xi$ be an $(h-1)$-dimensional subspace of $\Pi \simeq \PG(kh-1, q)$ disjoint from $\cS_{k-1, h-1}(q)$, and $\overline{\Xi}$ its extension in $\PG(kh-1, q^h)$. Let $\cL(\Theta)$ be a Desarguesian $(h-1)$-spread of $\Pi$ such that $\Xi \in \cL(\Theta)$ and $\cR_{h, q} \subset \cL(\Theta)$. Let $P$ denote the point $\Theta \cap \overline{\Xi}$. Then $\cL(\Theta)$ is the unique Desarguesian $(h-1)$-spread of $\Pi$ containing $\Xi$ and $\cR_{h, q}$ if and only if $\overline{\Xi} \cap \cS_{k-1, h-1}(q^h) = \left\{P, P^\psi, \dots, P^{\psi^{h-1}} \right\}$.
\end{cor}
\begin{proof}
Let $\cL(\Phi)$ be a further Desarguesian $(h-1)$-spread of $\Pi$ such that $\Xi \in \cL(\Phi)$ and $\cR_{h, q}~\subset~\cL(\Phi)$. Then $\Phi \in \cR_{k, q^h}$ and $\Phi \ne \Theta^{\psi^i}$, $i = 0, \dots, h-1$. Hence $\Phi \cap \Theta^{\psi^i} = \emptyset$, $i = 0, \dots, h-1$, and $Q = \overline{\Xi} \cap \Phi$ is a point of $\cS_{k-1, h-1}(q^h)$, with $Q \notin \left\{ P, P^{\psi}, \dots, P^{\psi^{h-1}}\right\}$. Similarly, if $Q$ is a point of $\overline{\Xi} \cap \cS_{k-1, h-1}(q^h)$, with $Q \notin \left\{ P, P^{\psi}, \dots, P^{\psi^{h-1}}\right\}$, let $\Phi$ be the unique member of $\cR_{k, q^h}$ through $Q$. Then $\cL(\Phi)$ is a Desarguesian $(h-1)$-spread of $\Pi$ such that $\Xi \in \cL(\Phi)$ and $\cR_{h, q}~\subset~\cL(\Phi)$. Moreover $\cL(\Phi) \ne \cL(\Theta)$.  
\end{proof}

We are now in a position to state and prove the main result of this subsection,
which provides a complete geometric characterization of when a Desarguesian
spread contains, besides the $(h-1)$-regulus of a Segre variety, an additional
$(h-1)$-dimensional subspace disjoint from it.

\begin{theorem}\label{thm:theta}
Let $\Xi$ be an $(h-1)$-dimensional subspace of $\Pi \simeq \PG(kh-1, q)$ disjoint from $\cS_{k-1, h-1}(q)$, and $\overline{\Xi}$ its extension in $\PG(kh-1, q^h)$. Let $\cL(\Theta)$ be a Desarguesian $(h-1)$-spread of $\Pi$ such that $\Xi \in \cL(\Theta)$ and $\cR_{h, q} \subset \cL(\Theta)$. Let $P$ denote the point $\Theta \cap \overline{\Xi}$. Then $\cL(\Theta)$ is the unique Desarguesian $(h-1)$-spread of $\Pi$ containing $\Xi$ and $\cR_{h, q}$ if and only if $P$ is not contained in a proper canonical subgeometry of $\Theta$ containing $\theta_q$.
\end{theorem}
\begin{proof}
By Corollary \ref{cor:point}, $\cL(\Theta)$ is the unique Desarguesian $(h-1)$-spread of $\Pi$ containing $\Xi$ and $\cR_{h, q}$ if and only if $\overline{\Xi} \cap \cS_{k-1, h-1}(q^h) = \left\{P, P^\psi, \dots, P^{\psi^{h-1}} \right\}$, where $\overline{\Xi} = \left\langle P, P^{\psi}, \dots, P^{\psi^{h-1}} \right\rangle_{q^h}$ and $P$ is defined as in \eqref{rep_P}. By \eqref{rep_pi}, $\overline{\Xi}$ is defined by $\overline{\pi}_v$. Hence, if $Q \in \overline{\Xi} \cap \cS_{k-1, h-1}(q^h)$ there are $\alpha_i \in \F_{q^h}$, $i = 0, \dots, h-1$, such that $Q$ is represented by the matrix  
\begin{align*}
\diag(\alpha_1, \dots, \alpha_{h}) M_v    
\end{align*}
of rank one. 

Let $\cI= \{i \mid \alpha_i \ne 0\}$. Since $\psi$ fixes $\cS_{k-1, h-1}(q^h)$, we may assume, without loss of generality,  that $1 \in \cI$. Notice that since 
\begin{align*}
\rk\left(\diag(\alpha_1, \dots, \alpha_{h}) M_v\right) = 1
\end{align*}
we have that
\begin{align*}
i \in \cI \iff v^{\sigma^i}\in \langle v \rangle_{q^h}.
\end{align*}
On the other hand, if $v^{\sigma^i},v^{\sigma^j} \in \langle v \rangle_{q^h}$, then $v^{\sigma^{-i}},v^{\sigma^{i+j}}\in \langle v \rangle_{q^h}$.
Hence $\cI$ must correspond to an additive subgroup of the group of integers modulo $h$. Therefore $|\cI| = t$, for some divisor $t$ of $h$. If $t = 1$, then $Q \in \{P, P^\psi, \dots, P^{\psi^{h-1}}\}$ and $P \in \theta_q$. If $t > 1$ and hence $Q \notin \{P, P^\psi, \dots, P^{\psi^{h-1}}\}$, let $h = rt$. Then $v^{\sigma^r} = \alpha v$ for some $\alpha \in \F_{q^h}$ and hence 
\[ v = v^{\sigma^h} = (v^{\sigma^r})^{\sigma^{r(t-1)}} =  \left(\left((v^{\sigma^r})\right)^{\sigma^r}\dots\right)^{\sigma^r} = \alpha^{1+q^r+\cdots+q^{r(t-1)}}v = \N_{q^h|q^r}(\alpha) v,
\]
where $\N_{q^m|q^r}$ denotes the field norm function from $\F_{q^h}$ to $\F_{q^r}$. It follows that $\N_{q^h|q^r}(\alpha)=1$, and so there exists (by Hilbert's Theorem 90) some $\beta \in \F_{q^h}$ such that $\alpha = \beta^{q^r-1}$. Let $v'=\beta^{-1} v$. Then 
\begin{align*}
(v')^{\sigma^r} = \beta^{-q^r} \alpha v = \beta^{1-q^r} \alpha v'=v'.
\end{align*}
Therefore, the $k$ entries of $v'$ lie in $\F_{q^r}$ and $v= \beta v'$. Hence, the point $P \in \Theta$ represented by \eqref{rep_P} lies in $\theta_{q^r}$, the canonical $q^r$-order subgeometry of $\Theta$. 
\end{proof}

\section{The intersection of two Desarguesian spreads} \label{sec:intersection}

In this section, using the geometric setting and the results established in the
previous section, we are able to determine how two Desarguesian spreads
intersect.

Consider the field $\F_{q^{kh}}$, regarded as a vector space of dimension $kh$ over $\F_q$. For each $\alpha \in \F_{q^{kh}}^*$, let $g_\alpha \in \GL_{kh}(q)$ be given by $g_\alpha: x \in \F_{q^{kh}} \mapsto \alpha x \in \F_{q^{kh}}$. Then 
\begin{align*}
\tilde{G} = \left\{g_{\alpha} \in \GL_{kh}(q) \mid \alpha \in \F_{q^{kh}}^*\right\}
\end{align*} 
is a cyclic subgroup of $\GL_{kh}(q)$ isomorphic to $\F_{q^{kh}}^*$, and $\tilde{G}$ is regular on the vectors of $\F_{q^{kh}}^*$. The image $G$ in $\PGL_{kh}(q)$ of a conjugate of $\tilde{G}$ in $\GL_{kh}(q)$ is called Singer group of $\PG(kh-1, q)$. The following results follow from \cite[Corollary 2]{Jones} and \cite[Theorem 7.3]{Huppert}.

\begin{prop}
\begin{itemize}
\item A cyclic regular subgroup of $\PGL_{kh}(q)$ is a Singer group of $\PG(kh-1, q)$. 
\item These cyclic regular subgroups form a single conjugacy class in $\PGL_{kh}(q)$ consisting of the Singer groups of $\PG(kh-1, q)$. 
\item The normalizer of a Singer group in $\PGL_{kh}(q)$ has order $kh \frac{q^{kh}-1}{q-1}$.
\end{itemize}
\end{prop}

Let $\tilde{H}$ be the subgroup of $\tilde{G}$ of order $q^h-1$, that is
\begin{align*}
\tilde{H} = \left\{g_{\alpha} \in \GL_{kh}(q) \mid \alpha \in \F_{q^{h}}^*\right\}.
\end{align*} 
The orbits of $\tilde{H}$ on the vectors of $\F_{q^{kh}}$ form a vector space partition $\tilde{\cD}_h$ of $\F_{q^{kh}}$, that is a set of $h$-dimensional vector subspaces such that every non-zero vector of $\F_{q^{kh}}$ is contained in a unique member of $\tilde{\cD}_h$, and $\tilde{G}/\tilde{H}$ acts regularly on $\tilde{\cD}_h$.
The image $H$ in $\PGL_{kh}(q)$ of a conjugate of $\tilde{H}$ in $\GL_{kh}(q)$ is called Singer subgroup of $\PG(kh-1, q)$ of order $\frac{q^h-1}{q-1}$. The set $\cD_{h}$, consisting of $(h-1)$-dimensional subspaces of $\PG(kh-1,q)$ whose underlying vector spaces are elements of $\tilde{\cD}_h$ form a Desarguesian $(h-1)$-spread of $\PG(kh-1, q)$. Moreover, $H$ leaves invariant every element of $\cD_h$. For more information see \cite{Drudge, VdV}.

A {\em pseudo-arc} of $\PG(kh-1,q)$ is a set $\cP_{h}$ of $(h-1)$-dimensional subspaces such that every subset of $k$ of them spans the entire space. When $h=1$, this definition reduces to the classical notion of arcs in projective
spaces. We recall the following fundamental bound on the size of a pseudo-arc.

\begin{theorem} [see \textnormal{\cite[Theorem 4.11]{Thas1971}}] Let $\cP_h$ be a pseudo-arc of $\PG(kh-1,q)$. We have
\begin{align*}
|\cP_h| \le 
\begin{cases}
q^h + k & \mbox{ if } q \mbox{ is even,} \\
q^h + k - 1 & \mbox{ if } q \mbox{ is odd.} \\
\end{cases}
\end{align*}
\end{theorem}

A pseudo-arc $\cP_{h}$ is said to be {\em Desarguesian} if $\cP_h \subseteq \cD_h$, where $\cD_h$ is a Desarguesian $(h-1)$-spread of $\PG(kh-1, q)$. Note that, if $k >2$, the largest known pseudo-arc $\cP_h$ of $\PG(kh-1, q)$ is Desarguesian, it has size $q^h+1$ and arises from a normal rational curve of $\PG(k-1, q^h)$ by field reduction.

\begin{prop}
A pseudo-arc $\cP_h$ of $\PG(kh-1, q)$ is Desarguesian if and only if $\cP_h$ is fixed by a Singer subgroup of $\PG(kh-1, q)$ of order $\frac{q^h-1}{q-1}$ elementwise. 
\end{prop}
\begin{proof}
If $\cD_h$ is a Desarguesian $(h-1)$-spread of $\PG(kh-1, q)$ and $\cP_h \subseteq \cD_h$, then there is a Singer subgroup of $\PG(kh-1, q)$ of order $\frac{q^h-1}{q-1}$, say $H$, stabilizing $\cD_h$ elementwise. Therefore $\cP_h$ is fixed by $H$ elementwise. Viceversa, if $H$ is a Singer subgroup of $\PG(kh-1, q)$ of order $\frac{q^h-1}{q-1}$ fixing $\cP_h$ elementwise then the orbits of $H$ on points of $\PG(kh-1, q)$ form a Desarguesian spread $\cD_h$ such that $\cP_h \subseteq \cD_h$.  
\end{proof}

\begin{lemma}\label{lemma:arc}
Let $\cP_{h}$ be a pseudo-arc of $\PG(kh-1,q)$ of size $k+1$. Then there exists a unique Segre variety $\cS_{k-1, h-1}(q)$ such that $\cP_{h} \subset \cR_{h, q}$.
\end{lemma}
\begin{proof}
Let $\cP_{h} = \{\Sigma_1, \Sigma_2, \dots, \Sigma_k, \Sigma_{k+1}\}$. Every point $P$ in $\PG(kh-1, q) \setminus \left( \bigcup_{i = 1}^k \Sigma_i \right)$ lies on a unique $(k-1)$-dimensional subspace spanned by the points $P_1, \dots, P_k$, where $P_i \in \Sigma_i$, $i = 1, \dots, k$. Varying $P \in \Sigma_{k+1}$ a set of $\frac{q^h-1}{q-1}$ pairwise disjoint $(k-1)$-dimensional subspaces is constructed. Similarly, by starting from $h+1$ of these $(k-1)$-dimensional subspaces so obtained, that form a pseudo-arc of $\PG(kh-1, q)$, a set of $\frac{q^k-1}{q-1}$ pairwise disjoint $(h-1)$-dimensional subspaces arises. By \cite{TV} they are the $(k-1)$-regulus and the $(h-1)$-regulus of a Segre variety $\cS_{k-1, h-1}(q)$.
\end{proof}

Next, we observe that $\PGL_{kh}(q)$ is transitive on pseudo-arcs of $(h-1)$-dimensional subspaces in $\PG(kh-1,q)$ of size $k+1$ (see also \cite[Theorem 1.6]{Thas1971}). Let 
\begin{align}
\begin{pmatrix}
A_1 & A_2 & \dots & A_k 
\end{pmatrix} \label{blockmat}
\end{align}
be a block $1 \times k$ matrix over $\F_q$, where $A_i$ is a square matrix of order $h$ over $\F_q$, $i = 1, \dots, k$. Denote by $\Sigma_i$, $i = 1, \dots, k$, the $(h-1)$-dimensional subspace of $\PG(kh-1, q)$ whose underlying vector space is spanned by the rows of \eqref{blockmat} where 
\begin{align*}
A_j = 
\begin{cases}
I_h & \mbox{ if } j = i, \\
\0_h & \mbox{ if } j \ne i.
\end{cases}
\end{align*}
Here $I_h$ and $\0_h$ are the identity and zero matrix of order $h$, respectively. Let $\Sigma$ be an $(h-1)$-dimensional subspace of $\PG(kh-1, q)$ such that any $k$ of $\Sigma_1, \Sigma_2, \dots, \Sigma_k, \Sigma$ span the whole $\PG(kh-1, q)$. Then the underlying vector space of $\Sigma$ is spanned by the rows of \eqref{blockmat}, where $A_i \in \GL_h(q)$. Indeed, 
\begin{align*}
\rk \begin{pmatrix} 
A_1 & A_2 & \dots & A_k  \\
\0_h & I_h & \dots & \0_h \\
\vdots & \vdots & \ddots & \vdots \\
\0_h & \0_h & \dots & I_h 
\end{pmatrix} = 
\rk \begin{pmatrix} 
I_h & \0_h & \dots & \0_h \\
A_1 & A_2 & \dots & A_k  \\
\vdots & \vdots & \ddots & \vdots \\
\0_h & \0_h & \dots & I_h 
\end{pmatrix} = \ldots = 
\rk \begin{pmatrix} 
I_h & \0_h & \dots & \0_h \\
\0_h & I_h & \dots & \0_h \\
\vdots & \vdots & \ddots & \vdots \\
A_1 & A_2 & \dots & A_k 
\end{pmatrix} = hk.
\end{align*}
Therefore $\Sigma$ can be chosen in 
$
\frac{|\GL_h(q)|^k}{|\GL_h(q)|}
$
ways. 

Let $H$ be the stabilizer in $\PGL_{kh}(q)$ of each of the $(h-1)$-dimensional subspaces $\Sigma_1, \dots, \Sigma_k$. Note that a projectivity of $H$ is induced by a block matrix given by 
\begin{align}
\begin{pmatrix}
A_1 & \0_h & \dots & \0_h \\
\0_h & A_2 & \dots & \0_h \\
\vdots & \vdots & \ddots & \vdots \\
\0_h & \0_h & \dots & A_k
\end{pmatrix},  \label{blockmat1}
\end{align}
where $A_i \in \GL_h(q)$. Hence the centre $Z(H)$ of $H$ is $\left\{\rho I_{hk} \mid \rho \in \F_q^*\right\}$ and 
$
|H| = \frac{|\GL_h(q)|^k}{q-1}$.

Let $\Sigma_{k+1}$ be the $(h-1)$-dimensional subspace whose underlying vector space is spanned by the rows of \eqref{blockmat}, where $A_1 = A_2 = \ldots = A_k = I_h$. Then an element of $Stab_H(\Sigma_{k+1})$ is induced by \eqref{blockmat1} where $A_1 = A_2 = \ldots = A_k$. Moreover $Z(H) \subset Stab_H(\Sigma_{k+1})$ and 
$
Stab_H(\Sigma_{k+1}) \simeq \frac{\GL_h(q)}{Z(H)}$.

It follows that $H$ is transitive on the $(h-1)$-dimensional subspaces $\Sigma$ of $\PG(kh-1, q)$ such that any $k$ of $\Sigma_1, \Sigma_2, \dots, \Sigma_k, \Sigma$ span the whole $\PG(kh-1, q)$.

\begin{prop}\label{prop:segre}
Let $\cP_h$ be a pseudo-arc of $\PG(kh-1, q)$ of size $k+1$ and let $\cS_{k-1, h-1}(q)$ be the unique Segre variety such that $\cP_{h} \subset \cR_{h, q}$. If $\cD_h$ is a Desarguesian $(h-1)$-spread of $\PG(kh-1, q)$ such that $\cP_h \subset \cD_h$, then $\cR_{h, q} \subset \cD_h$.  
\end{prop}
\begin{proof}
Let $\cP_h = \{\Sigma_1, \dots, \Sigma_{k+1}\}$. It is enough to observe that the unique cyclic group of $Stab_H(\Sigma_{k+1})$ acting regularly on the points of each of the $(h-1)$-dimensional subspaces of $\cP_h$ is a Singer subgroup of $\PG(kh - 1, q)$ of order $\frac{q^h-1}{q-1}$ and fixes both $\cD_h$ and $\cS_{k-1, h-1}(q)$.  
\end{proof}

\begin{theorem}\label{thm:intersection}
Let $\cP_h$ be a pseudo-arc of $\PG(kh-1, q)$ of size $k+1$ and let $\cD_h$, $\cD_h'$ be distinct Desarguesian $(h-1)$-spreads of $\PG(kh-1, q)$ such that $\cP_h \subset \cD_h \cap \cD_h'$. Then $|\cD_h \cap \cD_h'| = \frac{q^{kr}-1}{q^r-1}$ for some proper divisor $r$ of $h$. Moreover, $\cD_h \cap \cD_h' = \cR^r_{h, q}$, where $\cR^r_{h, q}$ is the system of $(h-1)$-dimensional subspaces of a generalized Segre variety $\cS^r_{kr-1, h-1}(q)$. 
\end{theorem}
\begin{proof}
Let $\cP_h$ be a pseudo-arc of $\Pi \simeq \PG(kh-1, q)$ of size $k+1$ and let $\cL(\Theta) = \cD_h$, $\cD_h'$ be distinct Desarguesian $(h-1)$-spreads of $\Pi \simeq \PG(kh-1, q)$ such that $\cP_h \subset \cD_h \cap \cD_h'$. By Lemma \ref{lemma:arc}, there is a Segre variety $\cS_{k-1, h-1}(q)$ of $\Pi$ such that $\cP_h \subset \cR_{h, q}$ and by Proposition \ref{prop:segre}, $\cR_{h, q} \subseteq \cD_h \cap \cD_h'$. If $h$ is prime, then $\cD_h \cap \cD_h' = \cR_{h, q}$, by Theorem \ref{thm:theta}. Assume that $\cR_{h, q} \subset \cD_h \cap \cD_h'$ and let $\Xi \in \cD_h \cap \cD_h'$, with $\Xi \notin \cR_{h, q}$. With the same notation used in subsection \ref{sec:unique}, let $P = \overline{\Xi} \cap \Theta$. By Theorem \ref{thm:theta}, there exists a proper divisor $r$ of $h$ such that $h = rt$ and $\theta_{q^r}$ is the canonical $q^r$-order subgeometry of $\Theta$ containing $P$ and $\theta_q$. By Proposition \ref{prop:segre} and Theorem \ref{thm:general}, it follows that $\Phi \in \cD_h \cap \cD_h'$, if $\overline{\Phi} \cap \Theta$ is a point of $\theta_{q^r}$. Hence, with the same notation used in subsection \ref{sec:generalized-segre}, $\cR^r_{h, q} \subseteq \cD_h \cap \cD_h'$. Then two possibilities arise: either $\cD_h \cap \cD_h' = \cR^r_{h, q}$ or $\cR^r_{h, q} \subset \cD_h \cap \cD_h'$. If the latter case occurs, then the previous argument can be repeated, and the statement follows.      
\end{proof}

The result in Theorem \ref{thm:intersection} extends and generalizes \cite[Corollary 1.6]{Rottey}. It would be nice to determine the number of Desarguesian $(h-1)$-spreads of $\PG(kh-1, q)$ that pairwise intersect in precisely $\cR^r_{h, q}$. By Theorem \ref{thm:intersection} such a number equals $\frac{|G|}{|N|}$, where $G$ is the stabilizer in $Stab_{\PGL_{kh}(q)}(\cS^r_{kr-1, h-1}(q))$ of $\cR^r_{h, q}$ and $N$ is the normalizer of a Singer cyclic subgroup of order $\frac{q^h-1}{q-1}$ in $G$.

\bigskip
{\footnotesize
\noindent\textit{Acknowledgments.}
The research was supported by the Italian National Group for Algebraic and Geometric Structures and their Applications (GNSAGA--INdAM). }

\bibliographystyle{abbrv}
\bibliography{biblio}

\end{document}